\documentclass[article]{amsart}
\usepackage[utf8]{inputenc}

\usepackage{fullpage,hyperref,enumitem}
\usepackage{amsmath,amssymb,mathtools}
\usepackage{float}
\usepackage{pgf}
\usepackage{amsrefs}
\usepackage{amsthm}
\usepackage{mathrsfs}

\usepackage{algorithm}

\theoremstyle{plain}
\newtheorem{theorem}{Theorem}

\newtheorem{lemma}[theorem]{Lemma}
\newtheorem{proposition}[theorem]{Proposition}

\newtheorem{definition}[theorem]{Definition}

\DeclareMathOperator{\CS}{CS}
\DeclareMathOperator{\CSL}{CSL}
\DeclareMathOperator{\CSR}{CSR}

\usepackage{thmtools} 
\usepackage{thm-restate}

\title{Repetition in Permutation Wordle}

\author[Hiveley]{Aurora Hiveley}
\address[A.~Hiveley]{Department of Mathematics, Rutgers University, Piscataway, NJ 08854}
\email{\textcolor{blue}{\href{mailto:aurora.hiveley@rutgers.edu}{aurora.hiveley@rutgers.edu}}}

\begin{document}
\begin{abstract}
In a game of permutation wordle, a player attempts to guess a secret permutation in the fewest number of guesses possible. Previously, Samuel Kutin and Lawren Smithline \cite{kutin2024} introduced this game and proposed a strategy called cyclic shift, which they conjecture performs optimally. We continue our investigation of this conjecture by considering how information is obtained and, at times, repeated during a game of permutation wordle using an arbitrary strategy. This analysis includes several algorithms to construct a secret permutation which prompts inefficient repetition according to the player's strategy, as well as proofs of their efficacy.
\end{abstract}

\maketitle

\section{Introduction} \label{sec:intro}
Consider a modification of Josh Wardle's New York Times game, Wordle, where instead of attempting to guess a 5-letter word, the player attempts to guess a permutation on $[n]$. On each turn, the player guesses a permutation and is subsequently told which positions of their guess are correct. The process repeats until the player correctly guesses the secret permutation. The player's goal is to complete the game in the fewest number of guesses possible. This game was first introduced by Kutin and Smithline in 2024, and they proposed a strategy called \textit{cyclic shift} which they conjecture has optimal performance in a game of permutation wordle. Cyclic shift, which we abbreviate $\CS$ in this paper, generates the $k$-th guess $\gamma_k$ in a game of permutation wordle by locking in all correct entries from $\gamma_{k-1}$ and shifting each incorrect entry one position to the right, skipping over the locked in entries in the process.

Previously, Kutin and Smithline studied the connection between the number of guesses that a game using $\CS$ will last and the number of excedances of the secret permutation, which in turn is related to the Eulerian numbers. From there, Hiveley \cite{hiveley2025} proved optimality of cyclic shift for games ending in exactly three guesses. In this work, a strategy is formalized as a list $S = [s_1, s_2, \dots, s_n]$ where $s_i$ is a permutation of length $i$ for $1 \leq i \leq n$, and the next guess $\gamma_{k}$ is generated by considering the set $\mathcal{I}_{k-1}$ of incorrect entries from guess $k-1$ and permuting them according to the permutation $s_{|\mathcal{I}_{k-1}|}$. Note that in this language, $\CS = [[1], [2,1], [2,3,1], \dots, [2,3,\dots, n, 1]]$, and we will sometimes use the list-indexing notation $S[k]$ instead of $s_k$ to refer to the $k$-th component of a strategy. 

In constructing a strategy for permutation wordle, we make a few assumptions about what an efficient game looks like. In a game consisting of guesses $\gamma_1, \gamma_2, \dots, \gamma_r$, we assume for simplicity that $\gamma_1$ is always the trivial permutation $[1,2,\dots, n-1,n]$. We also assume that a player wants to maximize the amount of new information learned on each turn, and a guess which fails to do so is inherently inefficient. 
For example, say that $1 \in \mathcal{I}_1$, meaning that the entry 1 is incorrect on guess one and therefore \textit{not} in the first position in the secret permutation. If the player were to guess the entry 1 in position one later on in the game, say on guess number $t$, then $\gamma_t$ cannot possibly be the secret permutation, so it is intrinsically an inefficient guess. The player should guess the entry 1 in some other un-guessed position in an effort to gain new information. In this sense, we require that the strategy components used to permute incorrect entries between guesses are \textit{derangements}, i.e. lacking fixed points. This guarantees that each incorrect entry is moved to a new spot compared to the previous guess, but how can we be sure that the entry's new position is \textit{brand new} in the context of the game as a whole? Our work in this paper will study just that: what we know about repetition in the course of a game, and what this framework indicates about Kutin and Smithline's conjecture about the optimality of $\CS$. 
We will begin with a discussion of repetition within a game of permutation wordle in Section \ref{sec:rep}, including several algorithms to produce secret permutations that spur such repetition. We will then conclude by examining repetition for games using the strategy $\CSL$, which was previously defined and studied in \cite{hiveley2025}.

\section{Repetition in Inductive Strategies} \label{sec:rep}

In \cite{hiveley2025}, we previously introduced the idea of an inductively constructed strategy, which has some arbitrary permutation for the $n$-th component $S[n]$, but the first $n-1$ components are all rightward cyclic shifting. The motivation for this framework came from an attempt to prove the optimality of $\CS$ by induction. However, in \cite{hiveley2025} it was shown that for any derangement acting as $S[n]$, the average number of guesses needed before $\mathcal{J}_k$ (the set of \textit{correct} entries on guess $k$) is nonempty and induction takes over is actually constant for a fixed $n$. This means that in order to prove Kutin and Smithline's conjecture about the optimality of $\CS$, we will have to pursue an approach that is more complex than simply considering inductively constructed strategies at the ``highest level." 

Instead, we turn our attention to the relationship between a strategy's components. 
First, consider an example game using the strategy $S = [[1], [2, 1], [2, 3, 1],[2, 1, 4, 3], [3, 4, 5, 2, 1]]$ with the secret permutation $p = [4,1,5,2,3]$. The game will proceed as follows:
\[\begin{array}{lll}
   \gamma_1 = [1,2,3,4,5] & \quad & \mathcal{I}_1 = \{1,2,3,4,5\} \\ 
   \gamma_2 = [5,4,1,2,3] & & \mathcal{I}_2 = \{1,2,3\} \\ 
   \gamma_3 = [1,5,4,2,3] & & \mathcal{I}_3 = \{1,2,3\} \\ 
   \gamma_3 = [4,1,5,2,3] & & \mathcal{I}_4 = \emptyset \\ 
\end{array}\]

Notice that in this game, the entry 1 is guessed incorrectly in position one in both $\gamma_1$ \textit{and} $\gamma_3$. By the assumptions established in Section \ref{sec:intro}, this game is suboptimal in the sense that $\gamma_3$. There is no way that $\gamma_3$ could have been the secret permutation since we already knew that $p[1] \neq 1$, so the game could not have ended after three guesses, \textit{and} we wasted an opportunity to learn new information about which digit is in position one and which position the digit 1 occupies in $p$. So, this strategy can be improved because it fails to perform optimally when the secret permutation is $p = [4,1,5,2,3]$. 

After some experimental calculations using Maple, we observe that this repetition of incorrect information actually happens for \textit{every} possible strategy in our construction for at least one possible secret permutation. All strategies, that is, except for $\CS$. We begin by presenting the following lemma. Observe that while we consider only rightward cyclic shifting in the proof, leftward cyclic shifting is handled analogously.

\begin{lemma} \label{lem:nodupes}
    $\CS$ never duplicates incorrect information.
\end{lemma}

\begin{proof}
    This follows directly from Definition 3.1 and Proposition 3.3 in \cite{kutin2024}. Assume for the sake of contradiction that $\CS$ produces two different guesses in which an incorrect entry is guessed in the same position. Let these two guesses be $\gamma_s$ and $\gamma_t$, and without loss of generality let $s < t$. Let $\gamma_s(i) = \gamma_t(i) = j$ where element $j \in \mathcal{I}_s, \mathcal{I}_t$.

    Since $j \in \mathcal{I}_s$, entry $i$ is shifted once to the right in $\gamma_{s+1}$. For $\gamma_t(i) = j$ later on, entry $i$ must have skipped over every open position in the permutation to return to position $i$. However, this means that there is no open position where entry $i$ is correct, which contradicts either the secret permutation being a permutation at all or the correctness of a previously considered correct position.    
\end{proof}

As mentioned in Section \ref{sec:intro}, we assume that a strategy \textit{must} be inefficient if incorrect information is duplicated. In other words, if $\gamma_i[m] = \gamma_j[m]$ for $i < j$ where $m \in \mathcal{I}_i$, then the strategy may be improved by simply changing $\gamma_j[m]$. In this scenario, $\gamma_j$ is guaranteed to be incorrect since position $m$ is incorrect. In this sense, $\CS$ certainly satisfies the barest minimum criteria for an effective strategy. And, in fact, any other strategy fails to meet this criteria.

\begin{theorem} \label{thm:alldupe}
    Any strategy which is \textit{not} $\CS$ (and whose components are derangements) repeats information for at least one permutation.
\end{theorem}

To prove this claim, we present an algorithm which, given the final component $S[n]$ of an \textit{inductive} strategy, constructs a permutation which, when acting as the secret permutation in a game of permutation wordle, causes the strategy $S$ to duplicate incorrect information. We will call these permutations \textit{offending permutations}, denoted $\omega$. Once we have a construction for inductive strategies, we will extend our algorithm in Section \ref{sec:allstrats} to include any strategy whose components are derangements, thereby proving the full claim. 

\subsection{Displacement} \label{sub:disvec}

Recall that all games of permutation wordle begin with the trivial guess, $[1,2,\dots, n]$. Then in order for the strategy component $S[n]$ to be utilized in the game (without cyclic shifting induction immediately taking over and producing no repeated incorrect entries per Lemma \ref{lem:nodupes}), we need $\mathcal{I}_1 = \emptyset$, so the secret permutation must be a derangement of $[n]$. Since $\mathcal{I}_1 = \emptyset$, when $S[n]$ generates the next guess we get that $\gamma_2 = S[n]^{-1}$. Since $S[n]$ is a derangement, the inverse is also a derangement, and for the remainder of the algorithm we will use this derangement $S[n]^{-1}$ to construct an offending permutation. 

The simplest way to force the repetition of bad information in a game of permutation wordle is to impel an entry $i$ back into position $i$, meaning that for some $t > 1$, we have $\gamma_1[i] = \gamma_t[i] = i$. To do this, we introduce the notion of \textit{displacement} as follows: 

\begin{definition} \label{def:disp}
The displacement vector $d(\pi)$ of a permutation $\pi \in S_n$ is the list of length $n$ such that $d(\pi)[i] := \pi[i] - i $ for all $1 \leq i \leq n$.    
\end{definition}

In other words, the list entry $d(\pi)[i]$ is the number of times that the entry $i$ must be shifted rightward before entry $i$ is in position $i$. 
For example, the rightward cyclic shifting component $S[n] = [2,3, \dots, n,1]$ has a displacement vector of $[1,1,\dots,1]$. Similarly, the left shifting component $S[n] = [n, 1,2, \dots, n-1]$ has a displacement vector $[n-1,n-1,\dots,n-1]$. For a more interesting example, the permutation $[5,3,2,6,1,4]$ has a displacement vector $[4,1,5,2,2,4]$. 

Then, to construct an offending permutation, we will study the displacement vector of $\gamma_2 = S[n]^{-1}$, as this will tell us how many times each entry in $\gamma_2$ must be right-shifted before returning to their position in $\gamma_1$. Since $S[n]^{-1}$ is a derangement, the displacement vector must consist of all nonzero numbers as there are no fixed points. From here, we will divide our study into two cases: displacement vectors which contain at least one element $e$ such that $2 \leq e \leq n-2$, and displacement vectors consisting of only $e \in \{1,n-1\}$. 
We begin by presenting an outline of our construction for a strategy of length 4 with the intention of motivating each step of the algorithm before generalizing to larger strategy lengths.

\subsection{Vectors Containing 2} \label{sub:2s}

In the case of $n=4$, note that $2 = n-2$, so any permutation whose inverse has at least one displacement vector entry $e$ such that $2 \leq e \leq n-2$ actually has an entry which is precisely equal to 2. Let position $i$ in $\gamma_2$ be an inverse entry with displacement equal to 2. Say that $\gamma_2[i+1] \in \mathcal{J}_2$. (Note, of course, that all of this addition is modulo $n$, so entries which would ``fall off" the end of the permutation actually loop back around to the front.) Then to produce $\gamma_3$, we right shift all incorrect entries, which means that the entry $\gamma_2[i]$ is shifted past the correct entry in position $i+1$ and becomes $\gamma_3[i+2]$. But since $\gamma_2[i]$ had displacement of 2, $\gamma_2[i] = i+2$, and hence $\gamma_2[i] = \gamma_1[i+2]$ seeing as our first guess is always the trivial permutation. Here's the rub, since $\gamma_3[i+2] = \gamma_2[i] = \gamma_1[i+2]$, and $i+2 \notin \mathcal{J}_1$ since $\mathcal{J}_1 = \emptyset$. Hence the same entry was guessed incorrectly in position $i+2$ more than once!

So if $\gamma_2[i+1] \in \mathcal{J}_2$, then we know that an entry will repeat on the third guess, $\gamma_3$. It suffices to construct the rest of the offending permutation $\omega$ in a way that will result in a legal game. Certainly the repeated entries must be right shifted at least once more in order to guess the secret permutation and end the game, so for the sake of simplicity say that $\mathcal{J}_2 = \{i+1\}$ and thus all other entries are right shifted to form $\gamma_4$. It remains, then, to verify that each of these entries is in a new position which they have not occupied before so that $\gamma_4$ can legally end the game. The details of this construction are expounded in the upcoming proof of Proposition \ref{prop:2s}.

As an example, let $S[4] = [2,4,1,3]$. Then $S[4]^{-1} = \gamma_2 = [3,1,4,2]$, and $d(S[4]^{-1}) = [2,3,1,2]$. Then let $i = 1$ since the element 3 in position one of $S[4]^{-1}$ has displacement of 2. Next, we let $i+1 = 2 \in \mathcal{J}_2$, so the entry 1 in position two is a correct entry. Then $\gamma_3 = [2,1,3,4]$ since we right shift all entries except for the correct entry, 1. Note that entries 3 and 4 are guessed in positions three and four, respectively, but they were already guessed in those positions in $\gamma_1$, so we have duplicated bad information. Lastly, we let $\mathcal{J}_3 = \mathcal{J}_2 = \{2\}$ and right shift all other entries one final time to obtain $\gamma_4 = [4,1,2,3]$. Note that if the secret permutation is $[4,1,2,3]$, then this is a legal game of permutation wordle, and using the inductive strategy with $S[4] = [2,4,1,3]$ caused the duplication of incorrect information between $\gamma_1$ and $\gamma_3$.

Note, of course, that an offending permutation need not repeat incorrect information in $\gamma_3$, specifically. Especially for larger $n$ which may lead to games lasting a longer number of turns, there are many offending permutations whose first duplicated incorrect entry occurs in the fourth guess or later. However, we only need to produce \textit{one} offending permutation in order to prove our claim, and the most straight-forward way to do this is be forcing repetition as soon as possible.

We formalize this constructive algorithm as follows, noting that addition of indices happens mod $n$, so $i+ 1 = (i \mod n) + 1$, for example. This algorithm is also implemented in the second of two Maple packages linked in the Appendix as the procedure \texttt{construct(p,right)}.

\begin{algorithm} 
\caption{Offending Permutation Constructor For $2 \in D$} \label{alg:2s}
\begin{enumerate}
    \item [(0)] Given the $n$-th strategy component $S[n]$ of an inductive strategy, consider $D = d\left( S[n]^{-1} \right)$

    \item Initialize $K = \emptyset$. For each $i \in [n]$, if $D[i] = 2$ and $i-1 \notin K$, add $i$ to $K$. Stop when $|K| = n-3$, even if there are remaining 2's to consider.
    
    \textit{Special case:} If $D[1] = D[n] = 2$ and $D[j] \neq 2$ for all $j \in \{2,\dots,n-1\}$, then let $i = n$ so that $K = \{n\}$. 
    
    \item For each $i \in K$, let $\omega[i+1] = S[n]^{-1}[i+1]$.

    \item For each $j \in \{1,\dots, n\} \setminus K$, shift $S[n]^{-1}[j]$ rightward twice, skipping over the entries locked in during Step (2). In other words, if $L = \{1, \dots, n\} \setminus K$ where $|L| = m$, then $\omega [L[j+2]] = S[n]^{-1}[L[j]]$ for $1 \leq j \leq m$.
    
\end{enumerate}
\end{algorithm}

Observe that in Step (1) we required that $(i-1) \notin K$ before adding $i$ to $K$. Without this condition, if, for example, $D = [2,2,\dots,2]$, then $K = \{1,\dots, n\}$ and so $\omega = \gamma_2$, which does not produce any repetition as the game ends after two guesses. Then, to ensure that another guess is generated (which in turn causes the repetition of incorrect information), we must have at least three entries not in $K$, hence the other condition in Step (1). We conclude our analysis of displacement vectors containing the entry 2 by proving the aforementioned algorithm's efficacy:

\begin{proposition} \label{prop:2s}
    If $S[n]$ is a derangement of length $n$ such that $2 \in D = d\left( S[n]^{-1} \right)$, then Algorithm \ref{alg:2s} produces an offending permutation $\omega$ for the inductive strategy with $S[n]$.
\end{proposition}

\begin{proof}
    Let $S[n]^{-1}$ have a displacement vector with at least one 2. Let $\omega$ be as defined in Algorithm \ref{alg:2s}. In a game of permutation wordle where $\gamma_2 = S[n]^{-1}$, say that $\mathcal{J}_2 = K+1$, i.e., the set $K$ where each element is incremented by 1 (modulo $n$.) Then $\mathcal{J}_2 \neq \emptyset$ since there is at least one 2 in $D$, so $\gamma_3$ is generated by the strategy component $S[n - |K|]$, which is right-shifting since $S$ is an inductive strategy. Then all incorrect entries shift once rightward in $\gamma_3$. 

    For each index $i \in K$, the entry in position $i+1$ (immediately to the right) is in $\mathcal{J}_2$, so the entry $\gamma_2[i]$ is right-shifted into position $i+2$ in $\gamma_3$. Since $D[i] = 2$, we have that $\gamma_2[i] = \gamma_3[i+2] = i+2$, and so the entry $i+2$ is guessed in the same position as in $\gamma_1$, thereby repeating incorrect information since $\mathcal{I}_3 \subset \mathcal{I}_1$. 

    According to the construction of $\omega$, $\mathcal{J}_3 = \mathcal{J}_2$, so every index not in $\mathcal{J}_2 = K+1$ will be right-shifted one more time to form $\gamma_4 = \omega$. Then it suffices to show that $\gamma_4$ is a legal end state for the game, i.e. that there are no additional repeated incorrect entries on guess four. Each $i \in K$ ends in position $i+3$ or $i+4$ after skipping over one locked entry to go from $\gamma_2$ to $\gamma_3$, and either skipping over one more entry to go from $\gamma_3$ to $\gamma_4$ or not. Each element in a position $i \in K$ has only been guessed in position $i+2$ (on $\gamma_1$ and $\gamma_3$) or in position $i$ on $\gamma_2$. Either way, $i+3$ and $i+4$ must be new positions as long as $i \geq 4$. Note that the $i+4$ state is not obtainable for a permutation of length 4 since $|\mathcal{J}_2| \leq n-3 = 1$, so the entry $i$ cannot skip over more than one correct entry.

    For any $j \notin K \cup \mathcal{J}_2$, we right shift twice. Then the end position of $\gamma_2[j]$ is either $j+2$ (if no correct entries are skipped over), $j+3$ (if one correct entry is skipped over) or $j+4$ (if two correct entries are skipped over). Notice that since $j \notin K$, the displacement of $\gamma_2[j]$ is not equal to 2, so if $\gamma_2[j]$ ends in position $j+2$, this is legal. If $\gamma_2[j]$ ends in position $j+3$, then it must have skipped over a correct entry to go from $\gamma_2$ to $\gamma_3$ or to go from $\gamma_3$ to $\gamma_4$. If $\gamma_2[j]$ skipped over a correct entry after $\gamma_2$, then $j+1 \in \mathcal{J}_2$, but this means that $j \in K$, which is a contradiction. If $\gamma_2[j]$ ends in position $j+4$, it must have skipped over one correct entry after each of $\gamma_2$ and $\gamma_3$ since there are no adjacent entries in $\mathcal{J}_2$ by construction. But once again we have a contradiction as this would mean that $j+1 \in \mathcal{J}_2$. So the only problematic possibility is if $\gamma_2[j]$ ends in position $j+3$ after skipping over a correct entry on $\gamma_3$. However, $\gamma_2[j]$ cannot have previously been guessed in position $j+3$ since position $j+2 \in K$ and hence $j+1$ was an entry with displacement of 2, meaning $\gamma_1 [j+1] = \gamma_3 [j+1] = j+1$. The only way that $\gamma_2[j]$ could have previously been guessed in position $j+3$ is if it happened in $\gamma_2$, but we know that $\gamma_2[j]$ was in position $j$ on $\gamma_2$ by definition. Hence, $\gamma_2[j]$ ends in a legal position at the end of $\gamma_4$.  
\end{proof}

\subsection{Vectors Not Containing 2} \label{sub:no2s}

For a strategy of length 4, we note that the displacement vector of $S[n]^{-1}$ can only have 1, 2, and 3 as entries, so all strategies can either be handled by the process described above or by the approach which will be outlined in Section \ref{sub:1&n-1}. However, the displacement vector of $S[n]^{-1}$ for a general strategy component of length $n$ will have entries in $\{1,2,\dots,n-1\}$, meaning that it is possible for the displacement vector used to construct $\omega$ to not have any 2's. The case where all entries are either $1$ or $n-1$ is detailed in the next section, so for now we turn our attention to displacement vectors with some entry $e \in \{3,\dots,n-2\}$. We further generalize our approach from the previous section as follows:

\begin{algorithm}
\caption{Offending Permutation Constructor For $2 \notin D$} \label{alg:no2s}

\begin{enumerate}
    \item [(0)] Given the $n$-th strategy component $S[n]$ of an inductive strategy, consider $D = d\left( S[n]^{-1} \right)$. This time, let $\mu = \min_{e \in D \setminus \{1\}}$.

    \item Let $\iota = \min \{ x \mid D[x] = \mu \}$, and let $K = \{ \iota + k \mid 1 \leq k \leq \mu - 1 \}$. 
    
    \item For each $i \in K$, let $\omega[i] = S[n]^{-1}[i]$. 

    \item For each $j \in [n] \setminus K$, shift $S[n]^{-1}[j]$ rightward twice, skipping over the entries locked in during Step (2). In other words, if $L = [n] \setminus K$ where $|L| = m$, then $\omega[L[j+2]] = S[n]^{-1}[L[j]]$ for $1 \leq j \leq m$.

    Observe that in this case, the entries $\iota$ and $\iota-1$ are each shifted $\mu + 1$ spaces rightward, while all other entries are shifted twice rightward without skipping.
    
\end{enumerate}
\end{algorithm}

As an example, consider $S[6] = [3,4,6,5,1,2]$. In such a game, $S[6]^{-1} = \gamma_2 = [5,6,1,2,4,3]$, and thus $D = [4,4,4,4,5,3]$. Then $\mu = 3$ and $\iota = 6$, so $K = \{1,2\}$. By our construction, $\mathcal{J}_2 = \{1,2\}$, so $\gamma_3 = [5,6,3,1,2,4]$ and we observe that the entry 3 is guessed in the same incorrect position as it was in $\gamma_1$. Allowing $\mathcal{J}_2 = \mathcal{J}_3$ once again, we obtain $\gamma_4 = [5,6,4,3,1,2]$, and if $\gamma_4$ is the secret permutation, our game ends here. 

\begin{proposition} \label{prop:no2s}
    If $S[n]$ is a derangement of length $n$ such that $2 \notin D = d\left( S[n]^{-1} \right)$, then Algorithm \ref{alg:no2s} produces an offending permutation $\omega$ for the inductive strategy with $S[n]$.
\end{proposition}

\begin{proof}
Let $S[n]^{-1}$ have a displacement vector containing at least one entry in $\{3,\dots,n-2\}$ and containing no 2's. Let $\mu$ be the minimum non-1 entry in $D$, and let $\omega$ be the permutation constructed as described in the algorithm above. Once again, it suffices to show that an entry is repeated (so the permutation \textit{is} an offender) and that the final game state $\gamma_4 = \omega$ is legal. 

Let $\iota$ be a position between 1 and $n$ such that $D[\iota] = \mu$. Since $\mathcal{J}_2 = K$ consists of the $\mu-1$ entries to the right of position $\iota$ (and $\mu \geq 3$ guarantees that $\mathcal{J}_2 \neq \emptyset$) then $\gamma_3$ is generated by the right shifting component $S[n-|K|]$ since $S$ is an inductive strategy. Then entry $\gamma_2[\iota] = I$ is shifted $\mu$ positions to the right, and since $D[\iota] = \mu$ we have that $\gamma_3[I] = \gamma_1[I] = I$ and thus incorrect information is repeated during the course of a game.

Then, to form $\gamma_4$, the entry $I$ is shifted one more position to the right. Recall that $\mathcal{J}_2 = K$ so $|\mathcal{J}_2| = \mu -1$, and by our construction we have that $\mathcal{J}_3 = \mathcal{J}_2$ and therefore $|\mathcal{I}_3| = n - (\mu - 1)$. Since $3 \leq \mu \leq n-2$, we then have that $|\mathcal{I}_3| \geq 3 $, so the entry $I$ is shifted into a new position in $\gamma_4$ as it takes on a third position from $\mathcal{I}_2 = \mathcal{I}_3$. 

It remains to show that no other entry from a position $j \in \mathcal{I}_2$ ends at a repeated index. Since $\mathcal{J}_2$ consists of the $\mu -1$ entries immediately to the right of the position $\iota$, we have that the entries of $\mathcal{I}_2$ are the $n - \mu$ entries immediately before the position $\iota$. Observe that any entry $j$ two or more positions to the left of $\iota$ will be right shifted twice without skipping over any correct entries. For example, if $j = \iota -2$ then the entry in position $j$ is right shifted twice into position $\iota$, which is the position immediately left of all entries in $K$. Since no entry in $\gamma_2$ had a displacement of 2, it is impossible for any of these entries to end in their same position in $\gamma_1$, and because $|\mathcal{I}_2| \geq 3$ each of the positions occupied by these entries in $\gamma_3$ and $\gamma_4$ must be new, hence there is no repetition and the ending state is legal.

The only remaining entry to consider is the entry in position $\iota-1$, since this entry will be right shifted into position $\iota$ on $\gamma_3$ and then into position $\iota + \mu$ on $\gamma_4$. However, much like before we observe that $\gamma_2[\iota-1]$ cannot have been guessed in position $\iota + \mu$ on a previous guess since the entry $I = \gamma_2[\iota]$ was guessed there in $\gamma_1$ and $\gamma_3$, and we know that $\gamma_2[\iota-1]$ was guessed in position $\iota-1$ in $\gamma_2$, by definition. And of course, $\iota-1 \neq \iota + \mu$ since $\mu \leq n-2$. Then every entry ends in a new position, so $\gamma_4 = \omega$ is a legal end state/secret permutation.  
\end{proof}

\subsection{Displacement Vectors of Derangements} \label{sub:1&n-1}

Recall that the sum of all entries in the displacement vector must be a multiple of $n$, so if a displacement vector consists only of 1's and $(n-1)$'s, then the vector must be either $[1,\dots, 1]$, $[n-1, \dots, n-1]$, or it must have an equal number of 1's and $(n-1)$'s. 

The two cyclic shifting strategy components are discussed in greater detail in Section \ref{sub:cycshift}, so we will focus only on the latter case. If the displacement vector does have an equal number of 1's and $(n-1)$'s, then the vector must have the form $[1,n-1,1,n-1,\dots, 1,n-1]$ or vice versa, otherwise two entries have displacements calculated from the same position, which is not possible. Such a permutation is actually a derangement consisting of $\frac{n}{2}$ 2-cycles. For example, the permutation $[6,3,2,5,4,1]$ has a displacement vector $[1,5,1,5,1,5]$ and consists of the two cycles $(16)(23)(45)$ when written in cycle notation. Observe also that these permutations are their own inverses, so $\gamma_2 = S[n]$. Then to prove Theorem \ref{thm:alldupe}, it suffices to show that we can still construct a repeating permutation when $S[n]$ is a derangement of this form.
 
We begin with an example for $n=4$. If $S$ is a strategy of length 4 with $S[4] = [2,1,4,3]$, then $S$ will repeat incorrect information when the secret permutation is $[3,4,1,2]$. Additionally, such a game can never terminate since $\gamma_1 = [1,2,3,4]$, and $\mathcal{J}_1 = \emptyset$ so $\gamma_2 = [2,1,4,3]$, but then $\mathcal{J}_2 = \emptyset$, so $\gamma_3 = [1,2,3,4] = \gamma_1$, and so on. 
\[ 
\begin{array}{ccc}
    \gamma_1 = [1,2,3,4] & & \mathcal{J}_1 = \emptyset \\
    \gamma_2 = [2,1,4,3] & & \mathcal{J}_2 = \emptyset \\
    \gamma_3 = [1,2,3,4] & & \mathcal{J}_3 = \emptyset \\
    \gamma_4 = [2,1,4,3] & & \mathcal{J}_4 = \emptyset \\
    \vdots & & \vdots
\end{array} 
\]

Notice that our guessing game enters an infinite loop wherein $S[4]$ repeatedly produces the same two guesses with no correct entries. Since $\mathcal{J}_s$ will always be empty for any $s \geq 1$, we will never be able to use a different strategy component to produce a different guess. Then the number of guesses needed to correctly guess $[3,4,1,2]$ is infinite. 
A similar argument applies to derangements $[2,3,1,5,6,4]$ or $[2,1,4,3,6,5]$ when used as $S[6]$, and so on. We generalize this observation as follows:

\begin{theorem}
    Let $\delta$ be a derangement of length $n$ with cycle type $[t_1, t_2, \dots, t_k]$. Let $\mu_i := |\{ j \mid t_j = t_i \}|$, or in other words $\mu_i$ is the multiplicity of the cycle length $t_i$ in the cycle type of $\delta$. If $\mu_i \geq 2$ for all $1 \leq i \leq n$, then any strategy with $S[n] = \delta$ will enter an infinite loop for at least one possible secret permutation $\omega$ of length $n$.
\end{theorem}

\begin{proof}
Let $\delta$ be a derangement with with cycle type $[t_1, t_2, \dots, t_k]$ such that $\mu_i \geq 2$ for all $1 \leq i \leq k$. Let $\delta = c_1 c_2 \dots c_k$ be the partition of $\delta$ into cycles where the length of $c_i$ is $t_i$ for each $i$. 
For a unique cycle length $t_i$, let $\mathcal{C}_i = \{c_i^1, c_i^2, \dots, c_i^{\mu_i} \}$ be the set of cycles in $\delta$ that each have length $t_i$. Then to construct an offending permutation $\omega$, map the elements of $c_i^j$ to the positions occupied by $c_i^{j+1}$ in $\delta$, in other words $\omega[c_i^j[a]] = c_i^{j+1}[a]$ for any $1 \leq a \leq t_i$. Repeat this mapping for all $1 \leq j \leq \mu_i$ (modulo $\mu_i$, of course, so that $c_i^{\mu_i}$'s elements map to the positions occupied by $c_i^1$) and for all unique cycle lengths $t_i$. Consider the example below where $\delta$ has length 10 and decomposes into two 2-cycles and two 3-cycles:
\[ 
\delta = [2,3,1,5,4,7,8,6,10,9] 
= (1,2,3)(4,5)(6,7)(8,9,10) 
\implies \omega = [8,9,10,6,7,4,5,1,2,3]
\]

Observe that a permutation wordle player using the inductive strategy such that $S[n] = \delta$ will enter an infinite loop while attempting to guess $\omega$. The player's first guess is always the trivial permutation $[1,2,\dots, n]$, and since $\delta$ was a derangement, each cycle $c_i$ in $\delta$ has length at least 2. Since $\mu_i \geq 2$ for all $i$, no entries of $\omega$ are shared with the trivial permutation. Then $\delta = S[n]$ generates the next guess as follows:
\[
\begin{array}{cc}   
\gamma_1 = \underbrace{1,2,\dots,t_1}_{c_1},\underbrace{t_1 + 1, \dots, t_1 + t_2}_{c_2}, \dots, \underbrace{n - t_k, \dots, n}_{c_k}  \\  
\gamma_2 = \underbrace{2,\dots,t_1,1}_{},\underbrace{t_1 + 2, \dots, t_1 + t_2, t_1+1}_{}, \dots, \underbrace{n - t_k + 1, \dots, n, n-t_k}_{}  \\
\end{array}
\]

Once again, none of the entries will be correct in $\gamma_2$ since each entry of $c_1$ was mapped to a location within the first $t_1$ locations, and the same for $c_2$, and so on. Since $\mathcal{J}_2 = \emptyset$, we generate $\gamma_3$ according to $\delta$ once again, but by the same logic as before, $\mathcal{J}_3 = \emptyset$. This process will repeat infinitely without ever locking in any correct entries, so we cannot guess the permutation $\omega$ in finitely many guesses using $S[n] = \delta$. 
\end{proof}

It is important to note that in the above proof, the length of the strategy $S$ need not be $n$. That is, the derangement $\delta$ may not be the longest component of $S$, and we can still identify a permutation that will cause this infinite looping later on in a game of permutation wordle. This concept will be studied in greater detail and broader generalization in Section \ref{sec:allstrats}.

With this in hand, we revisit our construction of offending permutations. To review, if a displacement vector consists of only 1's and $(n-1)$'s, then we showed that it must pairwise alternate between these entries. This means that the permutation $S[n]$ must be a derangement comprised of all 2-cycles, which in fact must be adjacent transpositions since each entry is either a distance of one leftward or one rightward away from its original position. Therefore $S[n]^{-1} = S[n]$, and if $\mathcal{J}_2 = \emptyset$ then $\gamma_3 = \gamma_1$. In this case, all information is repeated since all indices are incorrect and guessed in precisely the same positions as from the first guess. The game entering an infinite loop surely guarantees the repetition of incorrect information, so we construct our offending permutation such that the game enters an infinite loop. We have two possibilities using the construction described in the proof above:

\begin{algorithm}
\caption{Offending Permutation Constructor For $D \subset \{1,n-1\}^n$} \label{alg:1&n-1}
\begin{enumerate}
    \item If $S[n] = [2,1,4,3\dots, n,n-1]$, then $S[n] = (1,2)(3,4)\dots(n-1,n)$ when decomposed into cycles. Using the construction from above, we have the offending permutation $\omega = [3,4,\dots, n-1,n,1,2]$
    \item If $S[n] = [n,3,2\dots,n-1,n-2,1]$, then $S[n] = (1,n)(2,3)\dots(n-2,n-1)$ when decomposed into cycles. Using the construction from above, we have $\omega = [3,4,\dots, n-1,n,1,2]$
\end{enumerate}
\end{algorithm}

Observe that in either case, the same offending permutation is constructed! Then we have addressed the final sub-case for the displacement vector of $S[n]^{-1}$, and we are ready to generalize this construction to all strategies, not just inductive ones.

There is one case where $S[n]$ is a derangement which does not fall into this category, and it is the case where $S[4] = [3,4,1,2]$. This is a derangement consisting of the two 2-cycles $(1,3)(2,4)$, but the displacement vector takes the form $D = [2,2,2,2]$. If we employ Algorithm \ref{alg:2s}, then the secret permutation is constructed by setting $K = \{1\}$. Since $S[n]^{-1} = [3,4,1,2]$, we then have that $\omega[1] = 3$, and we right shift each of the other entries twice to obtain $\gamma_4 = \omega = [3,1,2,4]$. However, the entry 4 ends in position four, so this is not a legal end state for the game since that entry cannot be correct. The issue lies in the fact that an entry with displacement of two was right shifted two positions. In our previous proof and algorithm, this was not an issue, so we must provide an alternative construction for this case. Luckily, we can use a similar approach to the one employed in Algorithm \ref{alg:1&n-1} and let $\omega = [2,1,4,3]$. This secret permutation induces the same infinite loop incurred earlier in this section, so although $D$ does not take the form $[1,n-1,\dots,]$, this method still produces an offending permutation.

\subsection{Cyclic Shifting Components} \label{sub:cycshift}

There is are two possible displacement vectors which we omitted from our analysis thus far: the vectors $[1,1,\dots,1]$ and $[n-1,n-1,\dots,n-1]$. Observe that the latter is the case when $S[n]^{-1} = [n,1,2,\dots,n-1]$, which occurs when $S[n] = [2,3,\dots, n,1]$. But this is exactly cyclic shift, which never repeats incorrect information by Lemma \ref{lem:nodupes}. Then we need only concern ourselves with the former case, which results when $S[n]^{-1} = [2,3,\dots,n,1]$, i.e. when $S[n] = [n,1,2,\dots, n-1]$. An inductive strategy with such an $S[n]$ left shifts for the $n$-th component but right shifts for all other components, which has previously been dubbed $\CSL$ to denote the \textit{left} shifting in the $n$-th component \cite{hiveley2025}. 

To definitively prove Theorem \ref{thm:alldupe}, we must also construct at least one offending permutation for $\CSL$. As before, we note that $S[n] = [n,1,2,\dots,n-1]$, so making $\omega$ be a derangement so that $S[n]$ is used to generate at least one guess gives us $\gamma_2 = S[n]^{-1} = [2,3,\dots,n,1]$. Then all entries have displacement of 1, but we must have $\mathcal{J}_2 \neq \emptyset$ in order for right-shifting to take over. For simplicity, let $\mathcal{J}_2 = \{1\}$ so that the leading 2 in $\gamma_2$ is the only correct entry. Then $\gamma_3 = [2,1,3,\dots,n]$, and so the entries $\{3, \dots, n\}$ are all guessed in the same incorrect positions as they were in $\gamma_1$. Letting $\mathcal{J}_3 = \mathcal{J}_2 = \{1\}$ produces $\gamma_4 = [2,n,1,3,\dots,n-1]$, which is a legal final guess as long as $n \geq 4$. Each entry $j \in \{3, \dots, n-1\}$ ends in position $j+1$ after being guessed in position $j$ (in $\gamma_1$ and $\gamma-3$) and $j-1$ (in $\gamma_2$), entry 1 is guessed in position one and position $n$ before ending in position three, and entry $n$ is guessed in position $n$ and position $n-1$ before ending in position two. So as long as $n \geq 4$, each entry ends in a legal position. 

So what if $n=3$? Then $S = [[1],[2,1],[3,1,2]]$, which is the leftward cyclic shifting strategy. By Lemma \ref{lem:nodupes}, $S$ has no offending permutations, so we have considered all possible non-cyclic shifting strategies.

\section{Generalization to All Strategies} \label{sec:allstrats}

In Sections \ref{sub:2s} through \ref{sub:cycshift}, we outlined how to construct an offending permutation for an inductive strategy based upon the $n$-th strategy component, $S[n]$. We will now extend this framework to any strategy, not just inductive ones, in an effort to prove the stronger, overall claim that \textit{any} strategy (other than $\CS$) repeats incorrect information for at least one permutation.

Recall that the components of a strategy $S$ are required to be derangements since fixed points would cause incorrect entries to be guessed in the same incorrect position on consecutive guesses. Of course, $S[1] = [1]$, but this permutation is only included for the sake of completeness and standardization in our Maple code, it will never actually be utilized in a game of permutation wordle since it is not possible to have only one incorrect entry. Then $S[2] = [2,1]$ since there are only two permutations of length two: either $[1,2]$ (which consists exclusively of fixed points) or $[2,1]$ (the derangement). Note that any strategy has these same $S[1]$ and $S[2]$ components, so for an arbitrary strategy $S$ we will define $\kappa := \min \{ k \mid S[k] \neq \CS[k] \}$. Note that $\kappa \geq 2$ by the previous explanation, and of course $\kappa \leq n$ for a strategy of length $n$.

To construct an offending permutation, we want to utilize our previous work with inductive strategies, if possible. Observe that the first $\kappa$ components of $S$ form an inductive strategy of length $\kappa$. Then we will construct $\omega$ by fixing the permutation entries $\kappa+1, \kappa+2, \dots, n$ so that after guessing the trivial permutation for $\gamma_1$, we have that $\mathcal{I}_1 = \{ 1,2,\dots, \kappa \}$ and we can then use the construction from Sections \ref{sub:2s}-\ref{sub:1&n-1}. In other words, the offending permutation $\omega$ is made up of a sub-permutation of length $\kappa$ constructed according to the algorithms discussed thus far, and a sub-permutation of length $n - \kappa$ consisting exclusively of fixed points $[\kappa+1, \dots, n]$. Then we can generalize our construction to produce at least one offending permutation for any strategy $S$ (which is \textit{not} cyclic shift.)

Unfortunately, there is one edge case to consider. Say that $\kappa = 3$. Then the only derangement of length three which is not right-cyclic shifting is the permutation $[3,1,2]$, which corresponds to left-cyclic shifting as discussed in Section \ref{sub:cycshift}. However, there is no algorithm to construct an offending permutation for this case, and in fact the strategy $S = [[1],[2,1],[3,1,2]]$ has no offending permutations at all since it is cyclic shifting, and by Lemma \ref{lem:nodupes} any cyclic shifting strategy has no offending permutations. So we must generalize our offending permutation construction when $S[3] = [3,1,2]$, i.e., when inductive cyclic shifting occurs leftward. To do this, we must flip our construction so that it accommodates the leftward shifting later in the game. Then we define an analog to displacement: the \textit{left displacement}, as follows:

\begin{definition} \label{def:leftdisp}
    The left displacement vector $d_\ell(\pi)$ of a permutation $\pi \in S_n$ is the list of length $n$ such that $d_\ell(\pi)[i] := i - \pi[i]$ for all $1 \leq i \leq n$.  
\end{definition}

Then when an entry in an inductive strategy is correct, we left shift the incorrect entries backwards into positions of repetition, as before. The overall approach outlined in Algorithms \ref{alg:2s}-\ref{alg:1&n-1} therefore still holds so long as Step 3 shifts leftward rather than rightward. 

The only other modification we must make is to the approach used for $\CSL$ in Section \ref{sub:cycshift}. In this section, for the strategy $\CSL = [[1],[2,1], \dots, [2,\dots, n-1,1], [n,1,\dots, n-1]]$, we had the offending permutation $\omega = [2,n,1,3,\dots,n-1]$. The analog for left-cyclic shifting is the strategy $\CSR$ defined like so: 
\[ \CSR = [[1],[2,1],[3,1,2],\dots,[n-1,1,\dots,n-2],[2,\dots,n,1]]\] 

As before, we let $\mathcal{J}_1 = \emptyset$ so that $\gamma_2 = [n,1,2,\dots,n-1]$, and we then let $\mathcal{J}_2 = \{1\}$ so that after all entries in positions $\{2,\dots,n\}$ are left shifted, we obtain $\gamma_3 = [n,2,\dots, n-1,1]$. Here we see that entries $\{2,\dots,n-1\}$ are repeated in a familiar manner, and left-shifting the remaining entries yields $\gamma_4 = [n,3,\dots,n-1,1,2]$, which is a legal end state for the game. Then we have the following (modified) algorithm for an inductive strategy whose final component cyclic shifts in the opposite direction:

\begin{algorithm}[h!]
    \caption{Offending Permutation Constructor For $\CSL$ and $\CSR$} \label{alg:csl&csr}
    \begin{enumerate}
        \item If $S = \CSL =  [[1],[2,1], \dots, [2,\dots, n-1,1], [n,1,\dots, n-1]]$, then $\omega = [2,n,1,3,\dots,n-1]$
        \item If $S = \CSR = [[1],[2,1],[3,1,2],\dots,[n-1,1,\dots,n-2],[2,\dots,n,1]]$, then $\omega = [n,3,\dots,n-1,1,2]$
    \end{enumerate}
\end{algorithm}

With these modifications, the algorithm to construct an offending permutation for a general strategy is as follows:

\begin{algorithm}[h!]
\caption{Offending Permutation Constructor For a General Strategy $S$} \label{alg:GEN}
\begin{enumerate}
\item Check $S[3]$ to determine whether our base strategy is left or right shifting. If $S[3]=[2,3,1]$, then our base is right shifting. If $S[3]=[3,1,2]$, then our base is left shifting.

\item Let $\kappa = \min\{k \geq 4 \mid S[k] \text{ is not cyclic shifting in the direction from Step 1} \}$

\item If Step 1 was right-cyclic shifting, let $D = d(S[k]^{-1})$ and use casework from Sections \ref{sub:2s}-\ref{sub:cycshift} to construct $\omega'$. 

If Step 1 was left-cyclic shifting, let $D = d_\ell (S[k]^{-1})$ and construct $\omega'$ using casework from Sections \ref{sub:2s}-\ref{sub:1&n-1} (adjusted to left shift on Step 3) and from the modification of Section \ref{sub:cycshift} in Algorithm \ref{alg:csl&csr} in the preceding paragraph.

\item The offending permutation is $\omega = [\omega'[1], \omega'[2], \dots, \omega'[k], k+1, \dots, n] $ 
\end{enumerate}
\end{algorithm}

This construction is also implemented in the second Maple package linked in the appendix as the procedure \texttt{constructgen(s)}.

\section{Conclusion} \label{sec:conc}

In this paper, we presented an overall algorithm that takes an input strategy for a game of permutation wordle and outputs a permutation which, when acting as the secret permutation in a game using that strategy, causes incorrect information to be duplicated. Such duplication is inherently suboptimal, so the fact that this algorithm works for \textit{any} strategy other than cyclic shift leads us to the conclusion that $\CS$ is optimal in this sense. Now, it is important to note that any given strategy may have more than one offending permutation, and in fact there is often more than one. Counting and characterizing these permutations for a fixed strategy remains an open area of investigation.

For instance, in \cite{hiveley2025}, the sub-optimality of $\CSL$ was studied specifically for games ending in precisely three guesses. In fact, in this context $\CSL$ was proven to be definitively the worst inductive strategy in the sense that it guesses the \textit{fewest} secret permutations in three guesses or fewer when compared to any other inductive strategy. Experimental evidence also suggests that $\CSL$ under-performs compared to all other inductive strategies when analyzing offending permutations. The number of offending permutations for the strategy $\CSL$ of length $n \geq 4$ conjecturally exceeds the number of offending permutations for any other inductive strategy of the same length, and the number of such permutations is counted by the sequence $4, 35, 244, 1813, 14740, \dots $ This sequence does not yet have an entry in Sloane's Online Encyclopedia of Integer Sequences nor does it have a known closed form expression. 

Futhermore, our strategies each fix every component, $S[k]$, through the course of a game. This standardizes our analysis and experimental computation, but it is important to note that it is possible to change the permutation used to permute $k$ entries in the middle of a game, particularly if two consecutive guesses each resulted in the same number of incorrect entries. Say, for example, that in a game of length 5 your first guess was the trivial permutation $\gamma_1 = [1,2,3,4,5]$, but $\mathcal{J}_1 = \emptyset$. Say that you right-shifted to obtain $\gamma_2 = [5,1,2,3,4]$. If $\mathcal{J}_2 = \emptyset$ again, you may decide that this right-shifting strategy just isn't for you, and you could choose to use a different permutation to permute your five incorrect entries and produce $\gamma_3$. Under our formulation, $S[5]$ is fixed, so this would never happen, but it is important to note that in a real game this isn't forbidden.

\section*{Appendix} \label{sec:appen}
The majority of the findings in this paper are supported by two Maple packages, which are linked \href{https://aurorahiveley.github.io/wordle.txt}{here} and \href{https://aurorahiveley.github.io/wordlerep.txt}{here}. The first package was written for \cite{hiveley2025} and has not been modified from its original published state to fix a bug in an error checking procedure. The second package was written to support this paper, specifically, and it includes examples as well as a handful of procedures which verify (experimentally) each of the theorems presented in this paper. Any bugs should be reported to the author at aurora.hiveley@rutgers.edu.

\section*{Acknowledgements}
The author thanks her advisor Dr. Doron Zeilberger for the introduction to the problem and feedback on earlier drafts.

\bibliographystyle{plain}
\bibliography{sources}

\end{document}